\documentclass{amsart}
\usepackage{t1enc,amsfonts,commath,latexsym,amsmath,amsthm,verbatim,amssymb,graphics,mathdots,mathtools,pstool,color,wrapfig, enumitem}

\usepackage{cite}

\renewcommand{\Re}[1]{\mathsf{Re}(#1)}
\renewcommand{\Im}{\mathsf{Im}}

\newcommand{\R}{\mathbb{R}}
\newcommand{\C}{\mathbb{C}}

\newcommand{\HW}[1]{{\color{blue}#1}}

\def\BibTeX{{\rm B\kern-.05em{\sc i\kern-.025em b}\kern-.08em
		T\kern-.1667em\lower.7ex\hbox{E}\kern-.125emX}}

\newtheorem{theorem}{Theorem}[section]
\newtheorem{proposition}[theorem]{Proposition}
\newtheorem{lemma}[theorem]{Lemma}

\title{On joint eigen-decomposition of matrices}

\author[Troedsson et. al.]{Erik Troedsson$^{1}$, Daniel Falkowski$^1$, Carl-Fredrik Lidgren$^1$, Herwig Wendt$^2$, Marcus Carlsson$^{1}$}
\address{\noindent$~^1$Centre for Mathematical Sciences , Lund University, Lund,  Sweden.\newline
\indent$~^2$CNRS, IRIT, Universit\'{e} de Toulouse, Toulouse, France.}
\email{erik.troedsson@math.lu.se, marcus.carlsson@math.lu.se,\newline \indent herwig.wendt@irit.fr}

\begin{document}
\maketitle
\begin{abstract}
The problem of approximate joint diagonalization of a collection of matrices arises in a number of diverse engineering and signal processing problems. This problem is usually cast as an optimization problem, and it is the main goal of this publication to provide a theoretical study of the corresponding cost-functional. As our main result, we prove that this functional tends to infinity in the vicinity of rank-deficient matrices with probability one, thereby proving that the optimization problem is well posed. Secondly, we provide unified expressions for its higher-order derivatives in multilinear form, and explicit expressions for the gradient and the Hessian of the functional in standard form, thereby opening for new improved numerical schemes for the solution of the joint diagonalization problem. A special section is devoted to the important case of self-adjoint matrices.
\end{abstract}
\section{Introduction}
Let $\mathbb{M}(n,\mathbb{F})$ denote the set of $n\times n$ matrices, where $\mathbb{F}$ equals $\mathbb{R}$ or $\mathbb{C}$, and let $\mathcal{A} = \{A_1,\ldots,A_K\}$ be a collection of diagonalizable matrices. These matrices share a joint set of eigenvectors if and only if there exists an invertible matrix $Q$ such that $Q^{-1}A_kQ$ is a diagonal matrix, for each $k=1,\ldots,K$. In various applications there arises the situation that the matrices in $\mathcal{A}$ are jointly diagonalizable in theory, but not in practice due to noise and other imperfections, and it is then of interest to find the best matrix $Q$ that almost jointly diagonalize $\mathcal{A}$. This problem arises for instance in blind source separation \cite{comon2010handbook}, common principal component problem (CPC) \cite{flury1984common}, canonical polyadic decomposition of tensors \cite{luciani2014canonical}, multidimensional frequency estimation \cite{sahnoun2017multidimensional,andersson2018esprit} and the ``direction of arrival'' problem in three dimensions \cite{andersson2015method}.

To make the problem formulation rigorous, most prior works on the topic, such as \cite{andre2020joint,andre2023sparsity,cardoso1993blind,gong2012complex,iferroudjene2009new,luciani2014canonical,mesloub2018efficient}, introduce the cost-functional
\begin{equation}\label{functional}
    f_\mathcal{A}(Q)=\frac{1}{2}\sum_{k=1}^K\sum_{i\neq j} \left|(Q^{-1}A_kQ)_{ij}\right|^2 
\end{equation}
defined on $GL(n,\mathbb{F})$, the subset of invertible matrices. Clearly, $f(Q)=0$ if and only if $Q$ diagonalizes all matrices $A_k$ simultaneously. 

Despite its common usage, $f_\mathcal{A}$ has not been studied deeply on its own and many important properties are still unknown. So far, the main theoretical contribution that can be found in the literature is a computation of the gradient of \eqref{functional}, see \cite{hori1999joint}. 
Many interesting questions of large practical importance remain open, including the expressions for the Hessian and higher order derivatives, the local geometry of the functional, whether it has rank-deficient local minimizers, or under what conditions there is a unique global minimizer.

This work provides answers to some of these questions. 
The main achievements are, first, a theorem that shows that in general the functional will tend to infinity in the vicinity of all rank-deficient points, and second, unified expressions for the higher order derivatives of $f_\mathcal{A}$ along with expressions  in standard form the first and second derivative (gradient and Hessian). The former ensures that suitably engineered optimization routines (i.e., descent algorithms) cannot diverge, or cannot converge to rank-deficient points, whereas the latter are useful results for the construction of new descent based optimization routines, such as (conjugate) gradient descent or (quasi)-Newton algorithms (some of which are studied separately elsewhere, see e.g.~\cite{ctw2024joint}).


The paper is organized as follows: In Section \ref{sec2} we prove the main theorem about behavior near rank deficient points, and in Section \ref{sec3} we prove that this theorem applies in a concrete setting with probability one, given that the matrix entries in $\mathcal{A}$ are drawn from an absolutely continuous distribution. In Section \ref{sec4} we prove a general formula for all higher order derivatives, expressed as multi-linear forms, and in Sections \ref{sec5} and \ref{sec6} we use this to derive more concrete representations of the gradient and Hessian. Finally, in Section \ref{sec7}, we extend our key results to the special case of self-adjoint matrices, which is of importance in certain applications, and briefly discuss particularities thereof.

\section{Behaviour near rank-deficient points}\label{sec2}

Despite its common usage as measure of joint diagonality, there seems to exist no study of the basic properties of $f_{\mathcal{A}}$. Indeed, assuming that one has an algorithm designed to find a local minimizer of $f_{\mathcal{A}}$ (or some proxy thereof), the first question that comes to mind is if the formulation is mathematically sound, or whether it is possible that such an algorithm will converge to a rank-deficient matrix. In this section we show that, under some mild conditions on the matrices $\mathcal{A}$, the functional $f_\mathcal{A}(Q)$ tends to infinity whenever the argument converges to a rank-deficient matrix.

Recall that a (linear) subspace $\mathcal{M}$ of $\mathbb{F}^n$ is invariant for a matrix $A$ if $A(\mathcal{M})\subseteq\mathcal{M}$. For diagonalizable matrices, the invariant subspaces are precisely the subspaces that are spanned by any subset of eigenvectors. Naturally, two or more matrices are said to have a common invariant subspace if there is a subspace that is invariant for each of the matrices in question. The probability that two or more matrices have a 
common invariant subspace
is zero, as we will show in the next section, but first let us state and prove our main theorem.

\begin{theorem}\label{diverg}
    We have that
    $\lim_{Q\to Z}f_{\mathcal{A}}(Q) = \infty$ 
    for all fixed nonzero rank-deficient matrices $Z$, if and only if the matrices in $\mathcal{A}$ have no common nontrivial invariant subspace.
\end{theorem}

Let $I$ denote the identity matrix and $J$ the matrix with zeros on the diagonal and ones everywhere else. Also let $X\diamond Y$ denote the pointwise multiplication of two matrices $X$ and $Y$, also known as Hadamard multiplication. We can then express \eqref{functional} more compactly as
\begin{equation}
	f_\mathcal{A}(Q)=\frac{1}{2}\sum_{k=1}^K \|J\diamond (Q^{-1}A_k Q)\|^2,
\end{equation}
where the norm is the Frobenius norm.
To prove Theorem \ref{diverg}, we begin by noting a comparison between the norm $\|\cdot\|$ and the {off-}diagonal measure $\|J\diamond \cdot\|$ appearing in $f_\mathcal{A}$.

\begin{lemma}\label{lem1}
	Let $A\in \mathbb{M}(n,\mathbb{C})$. Then
	\begin{equation}
		\|J\diamond A\| \leq \|A\| \leq \sqrt{n}(\max_{\lambda\in\sigma(A)} |\lambda|+2n\|J\diamond A\|),
	\end{equation}
	where $\sigma(A)$ is the set of eigenvalues of $A$.
\end{lemma}
\begin{proof}
	The inequality $\|J\diamond A\| \leq \|A\|$ is obvious. To prove the second one, define the ``Ger\v{s}gorin disks'',
	$$D_i \coloneqq D\big(a_{ii}, \sum_{j\neq i} |a_{ij}|\big) = \{z\in\mathbb{C}; |a_{ii}-z| \leq \sum_{j\neq i} |a_{ij}|\},$$
where $D(x,r)$ denotes the closed disk with center $x$ and radius $r$.
	By the Ger\v{s}gorin circle theorem, we know that $\sigma(A) \subset \cup_{i=1}^n D_i$ and that each connected component of $\cup_{i=1}^n D_i$ contains at least one eigenvalue $\lambda$. 
	
	Consider the diagonal element $a_{kk}$ and the connected region $\Omega_k\subseteq\cup_{i=1}^n D_i$ that contains it. This region is a union of disks, $\Omega_k=\cup_{i\in I} D_i$ for some index set $I\subseteq \{1,2,\ldots,n\}$ with $k\in I$. Since these disks are connected, they are all contained in the larger disk centered at $a_{kk}$ of radius $2\sum_{j\in I} \sum_{i\neq j} |a_{ij}| $. Let $\lambda$ be an eigenvalue in $\Omega_k$ and note that $$|a_{kk}|\leq |\lambda|+|a_{kk}-\lambda|\leq |\lambda|+2\sum_{j=1}^n \sum_{i\neq j} |a_{ij}|\leq \max_{\lambda\in\sigma(A)}|\lambda|+2\sqrt{n(n-1)}\|J\diamond A\|,$$ where we used the Cauchy-Schwarz inequality in the final estimate.
Thus \begin{align*}
 & \|A\|^2=\sum_{i,j}|a_{ij}|^2 \leq \|J\diamond A\|^2+\sum_k|a_{kk}|^2\leq\\
       & \|J\diamond A\|^2+n\left(\max_{\lambda\in\sigma(A)}|\lambda|+2\sqrt{n(n-1)}\|J\diamond A\|\right)^2\leq\\& \|J\diamond A\|^2+ n\left((\max_{\lambda\in\sigma(A)}|\lambda|)^2+4n(\max_{\lambda\in\sigma(A)}|\lambda|)\|J\diamond A\|+4(n^2-n)\|J\diamond A\|^2\right)\leq \\n&\left(\max_{\lambda\in\sigma(A)}|\lambda|+2 n^2\|J\diamond A\|\right)^2
     \end{align*}
 The desired inequality follows upon taking the square root of the last line.
\end{proof}

We shall also need the following result, which explains when the norm of a similarity transformation $Q^{-1}AQ$ diverges.
\begin{lemma}\label{lem2}
		Let $ {A}, {Z}$ be $n\times n$ matrices such that $ {Z}\neq 0$ is rank-deficient. Then $$\underset{Q\in GL(n,\mathbb{C})}{\underset{Q\rightarrow Z}{\lim}} \|Q^{-1} A Q\| = \infty$$ if and only if
		$A(\text{Im}(Z))\nsubseteq \text{Im}(Z)$. 
\end{lemma}
\begin{proof}
First assume that there exists a sequence $(Q_j)_{j=1}^\infty$ in $GL(n,\mathbb{C})$ converging to $Z$ such that 
$\|Q_j^{-1} A Q_j\|$ stays bounded. By the singular value decomposition we can write $Q_j = U_j\Sigma_j V_j^*,$ and by a standard compactness argument we can extract subsequences of $(U_j)_{j=1}^\infty$ and $(V_j)_{j=1}^\infty$ that converge to some unitary matrices $U_{\infty}$ and $V_\infty$, respectively. Note that  $(V_j^*Q_j^*Q_j V_j)_{j=1}^\infty=( \Sigma_j^2)_{j=1}^\infty$ converges to $V_\infty^*Z^*Z V_\infty$, so the latter is a diagonal matrix that we can denote $\Sigma_{\infty}^2$. By taking the limit of $Q_jV_j=U_j\Sigma_j$ we also arrive at $ZV_\infty=U_\infty\Sigma_{\infty}$, from which it follows that $$Z=U_\infty\Sigma_{\infty}V_\infty^*,$$ i.e.~the triple $(U_\infty,\Sigma_{\infty},V_\infty)$ provides a singular value decomposition of $Z$.   
Now, since the Frobenius norm is unitarily invariant, the expression $\|Q_j^{-1}AQ_j\|$ simplifies to
	\begin{equation}\label{svdfro}
		\|Q_j^{-1}AQ_j\| = \|V_j\Sigma_j^{-1}U_j^*AU_j\Sigma_jV_j^*\| = \|\Sigma_j^{-1}U_j^*AU_j\Sigma_j\| = \|M_j\|,
	\end{equation}
	where we have set $M_j\coloneq \Sigma_j^{-1}U_j^*AU_j\Sigma_j$.
In order to estimate $M_j$, let $r=\text{rank}(Z)$ and decompose $\Sigma_j,U_j$ (including $j=\infty$) as block matrices:
	$$
	\Sigma_j = \begin{pmatrix}
		\Sigma_{1,j} & 0 \\
		0 & \Sigma_{2,j}
	\end{pmatrix}
	\qquad
	U_j = \begin{pmatrix}
		U_{1,j} & U_{2,j}
	\end{pmatrix},
	$$
	where $\Sigma_{1,j},\Sigma_{2,j},V_{1,j}$ and $V_{2,j}$ are of sizes $r\times r,(n-r)\times(n-r),n\times r$ and $n\times (n-r)$, respectively. Since $Q_j$, is invertible for $j\in\mathbb{N}$, so is $\Sigma_{1,j},\Sigma_{2,j}$. We then have
	\begin{equation}\label{bigmat}\begin{aligned}
		&M_j =\begin{pmatrix}
			\Sigma_{1,j}^{-1} & 0 \\
			0 & \Sigma_{2,j}^{-1}
		\end{pmatrix}
		\begin{pmatrix}
			U_{1,j}^* \\
			U_{2,j}^*
		\end{pmatrix}
		A
		\begin{pmatrix}
			U_{1,j} & U_{2,j}
		\end{pmatrix}
		\begin{pmatrix}
			\Sigma_{1,j} & 0 \\
			0 & \Sigma_{2,j}
		\end{pmatrix}= \\
		&\begin{pmatrix}
			\Sigma_{1,j}^{-1}U_{1,j}^*AU_{1,j}\Sigma_{1,j} & \Sigma_{1,j}^{-1}U_{1,j}^*AU_{2,j}\Sigma_{2,j} \\
			\Sigma_{2,j}^{-1}U_{2,j}^*AU_{1,j}\Sigma_{1,j} & \Sigma_{2,j}^{-1}U_{2,j}^*AU_{2,j}\Sigma_{2,j}
		\end{pmatrix} 
		=: \begin{pmatrix}
			M_j^{11} & M_j^{12} \\
			M_j^{21} & M_j^{22} 
		\end{pmatrix}.
	\end{aligned}
\end{equation}

Consider $M_j^{21}$ as a product of the matrices $\Sigma_{2,j}^{-1}$ and $U_{2,j}^*AU_{1,j}\Sigma_{1,j}$. Since $\Sigma_{2,j}\to 0$, the inverse $\Sigma_{2,j}^{-1}$ is a diagonal matrix where all diagonal elements tend to $\infty$. 
Henceforth, if we show that the other factor converges to a non-vanishing matrix, this would imply that $(\|M^{21}_j\|)_{j=1}^{\infty}$ (and hence also $(\|M_j\|)_{j=1}^\infty$) converges to infinity, contradicting our original assumption. Now, the other factor clearly
converges to $U_{2,\infty}^*AU_{1,\infty}\Sigma_{1,\infty}$, which thus has to be identically zero. But this matrix is zero if and only if 
\begin{equation}\label{kerineq1}
  U_{2,\infty}^*AU_{1,\infty}= 0.
\end{equation}
As the columns of $U_{1,\infty}, U_{2,\infty}$ form a basis for $\text{Im}(Z)$ and $\text{Im}(Z)^{\perp}$ respectively, the above is equivalent to \begin{equation}\label{kerineq2}A(\text{Im}(Z))\subseteq \text{Im}(Z),\end{equation}
as was to be shown.

Conversely, assume that \eqref{kerineq2} holds, let $U_\infty\Sigma_{\infty}V_\infty^*$ be an SVD of $Z$ and adopt the same notation as above. Set 
$$Q_j = U_\infty\begin{pmatrix}
	\Sigma_{1,\infty} & 0 \\
	0 & \frac{1}{j}I
\end{pmatrix}V_\infty^*,$$
where $I$ here is the $(n-r)\times(n-r)$ identity matrix. By repeating the computations in \eqref{bigmat}, we see that $M_j$ is given by
\begin{align*}
	&\begin{pmatrix}
		\Sigma_{1,\infty}^{-1}U_{1,\infty}^*AU_{1,\infty}\Sigma_{1,\infty} & \Sigma_{1,\infty}^{-1}U_{1,\infty}^*AU_{2,\infty}j^{-1} \\
		jU_{1,\infty}^*AU_{2,\infty}\Sigma_{1,\infty} & jU_{2,\infty}^*AU_{2,\infty}j^{-1}
	\end{pmatrix}=\\
	&\begin{pmatrix}
		\Sigma_{1,\infty}U_{1,\infty}^*AU_{1,\infty}\Sigma_{1,\infty}^{-1} & j^{-1}\Sigma_{1,\infty}^{-1}U_{1,\infty}^*AU_{2,\infty} \\
		0 & U_{2,\infty}^*AU_{2,\infty}
	\end{pmatrix}
\end{align*}
where the lower left corner is zero by the equivalence of \eqref{kerineq1} and \eqref{kerineq2}. Since the only dependence on $j$ is the factor $j^{-1}$ in the upper corner, the above construction provides a sequence $(Q_j)_{j=1}^\infty$ in $GL(n,\mathbb{C})$ such that \eqref{svdfro} is bounded, and the proof is complete.
\end{proof}

We are now ready to prove Theorem \ref{diverg}.
\begin{proof}[Proof of Theorem \ref{diverg}]
	Let $A$ and $Z$ be a fixed matrices where $Z$ is rank-deficient. Combining Lemma \ref{lem1} with the fact that the mapping
	$$A\to Q^{-1} A Q$$
	leaves the eigenvalues unchanged, we conclude that 
	$\lim_{Q\to Z} \|J\diamond \left(Q^{-1} A Q\right)\| = \infty$ if and only if $\lim_{Q\to Z} \| \left(Q^{-1} A Q\right)\| = \infty$. By Lemma \ref{lem2} this in turn happens
	if and only if 
	$A(\text{Im}(Z))\nsubseteq \text{Im}(Z),$
	i.e. $\text{Im}(Z)$ is not an invariant subspace of $A$. Since by definition
	\begin{equation}
		f_\mathcal{A}(Q)=\frac{1}{2}\sum_{k=1}^K \|J\diamond Q^{-1}A_k Q\|^2,
	\end{equation}
	we see that $\lim_{Q\to Z} f_\mathcal{A}(Q) =\infty$ if and only if $\text{Im}(Z)$ is not an invariant subspace for \textit{at least one} of the matrices $A_1,\ldots,A_K$. Thus, if $A_1,\ldots,A_K$ has no common nontrivial invariant subspace we get $\lim_{Q\to Z} f_\mathcal{A}(Q) =\infty$ for all rank-deficient matrices $Z\neq 0$. Conversely, if there is such a common subspace, say $\mathcal{Y}$, we can construct a matrix ${Z}$ with $\text{Im}({Z}) = \mathcal{Y}$ and a sequence ${Q}_j\to {Z}$ as in Lemma \ref{lem2} such that $\sup_j f_\mathcal{A}({Q}_j) < \infty$.
\end{proof}

\section{$\mathcal{A}$ has no nontrivial invariant subspace with probability one}\label{sec3}

In a concrete application, in order to show that  $\lim_{Q\to Z}f_{\mathcal{A}}(Q) = \infty$ for all nonzero rank-deficient matrices, one must be able to verify the assumptions on the matrix tuple $\mathcal{A}$: that it has no common nontrivial invariant subspace. Since this can be an exhaustive task in practice, we show in this section that this is satisfied for ``almost all'' matrix tuples $\mathcal{A}$. By ``the'' Lebesgue measure on $\mathbb{M}(n,\mathbb{F})$ we simply mean a Haar measure (which is unique up to a constant), or equivalently, the Lebesgue measure induced by identifying $\mathbb{M}(n,\mathbb{F})$ with $\mathbb{F}^{n^2}$ under some linear transformation. A probability measure $\mathcal{P}$ on the product space $\big(\mathbb{M}(n,\mathbb{F})\big)^K$ is said to be absolutely continuous if it is absolutely continuous with respect to the corresponding product Lebesgue-measure.

\begin{proposition}\label{probone}
Let $\mathcal{P}$ be an absolutely continuous probability measure on $\big(\mathbb{M}(n,\mathbb{F})\big)^K$. Then the set of tuples $\mathcal{A}=\{A_1,\ldots,A_K\}$ that have a common invariant subspace has measure zero. In particular, if $\mathcal{A}$ is drawn from $\mathcal{P}$, it holds with probability one that $\lim_{Q\to Z}f_{\mathcal{A}}(Q) = \infty$ at all non-zero rank-deficient points $Z$.
\end{proposition}
\begin{proof}
	Let $\mu$ denote regular Lebesgue measure on $\mathbb{M}(n,\mathbb{F})$. We first prove that given a fixed nontrivial subspace $\mathcal{Y}\subseteq \mathbb{F}^n$, the set of matrices invariant under $\mathcal{Y}$ has measure zero:
	$$\mu(\{A:~ A\mathcal{Y}\subseteq \mathcal{Y}\}) = 0.$$
	To this end, let $m=\operatorname{dim}(\mathcal{Y})$ and choose an orthonormal basis $v_1,\ldots,v_n$ of $\mathbb{F}^n$ such that the first $m$ vectors $v_1,\ldots,v_m$ form a basis for $\mathcal{Y}$. Then $A\mathcal{Y}\subseteq \mathcal{Y}$ if and only if
	$$\begin{pmatrix}
		v_{m+1} & \cdots & v_n
	\end{pmatrix}^T
	A
	\begin{pmatrix}
		v_{1} & \cdots & v_m
	\end{pmatrix}
	=0.$$
	This is a nontrivial linear equation system in the coefficients of $A$, so the set of solutions is a nontrivial linear subspace, which thus has Lebesgue measure $0$.
	
	We now turn attention to the product space $\mathbb{M}(n,\mathbb{F})\times \mathbb{M}(n,\mathbb{F})$ equipped with the product measure $\mu \times \mu$. Consider the set
	$$\mathcal{S} = \{(A,B):~ A,B\text{ have a common n.i.s.}\}, $$
	where ``n.i.s.'' is short for ``nontrivial invariant subspace''. Let $\chi_\mathcal{S}(A,B)$ be the indicator function for $\mathcal{S}$. By Tonelli's theorem,
	$$(\mu\times\mu)(\mathcal{S}) = \int\left(\int \chi_\mathcal{S}(A,B)\,d\mu(B) \right) d\mu(A)$$
	For fixed $A$ with distinct eigenvalues, the inner integral is equal to
	\begin{align*}
		\int \chi_\mathcal{S}(A,B)\,d\mu(B) &= \mu\left(\bigcup_{\{\mathcal{Y}\subseteq \mathbb{F}^n, \mathcal{Y} \text{ is a n.i.s. of A}\}}\{B:~ B\mathcal{Y}\subseteq \mathcal{Y}\}\right)\\
		&\leq \sum_{\{\mathcal{Y}\subseteq \mathbb{F}^n, \mathcal{Y} \text{ is a n.i.s. of A}\}} \mu\left(\{B:~ B\mathcal{Y}\subseteq \mathcal{Y}\}\right) =0
	\end{align*}
	since having distinct eigenvalues means that $A$ only has a finite number of invariant subspaces. Now, by Lemma \ref{distincteig} below (which, although it is standard, is included for completeness) the set of $A$'s which do not have distinct eigenvalues has Lebesgue measure $0$, so we get $(\mu\times\mu)(\mathcal{S})=0$ which proves the proposition in the case $K=2$. But if $\mathcal{A}$ has a common n.i.s.~for $K\geq 3$ there must be two matrices that share this n.i.s., which concludes the proof.
\end{proof}

\begin{lemma}\label{distincteig}
	The set of matrices in $\mathbb{M}(n,\mathbb{F})$ that have non-distinct eigenvalues has Lebesgue measure zero.
\end{lemma}
\begin{proof}
A matrix $A$ has $n$ distinct eigenvalues if and only if its characteristic polynomial, 
	$$p_A(z) \coloneqq \operatorname{det}(A-zI) = p_nz^n + p_{n-1}z^{n-1} + \cdots + p_0,$$
	has $n$ distinct roots. This in turn can only happen if $p_A(z)$ and its derivative $p'_A(z)= q_{n-1}z^{n-1} + \cdots + q_0$ have a common factor, which is equivalent to the Sylvester matrix,
	\[ S(p_A,p'_A) = 
	\begin{pmatrix}
		p_n & p_{n-1} & \cdots & p_0 & 0 & 0 & \cdots & 0 \\
		0 & p_{n} & \cdots & p_1 & p_0 & 0 & \cdots & 0 \\
		\cdots & \cdots & \cdots & \cdots & \cdots & \cdots & \cdots & \cdots \\
		0 & 0 & \cdots & 0 & p_n & p_{n-1} & \cdots & p_0 \\
		q_{n-1} & q_{n-2} & \cdots & q_0 & 0 & 0 & \cdots & 0 \\
		0 & q_{n-1} & \cdots & q_1 & q_0 & 0 & \cdots & 0 \\
		\cdots & \cdots & \cdots & \cdots & \cdots & \cdots & \cdots & \cdots \\
		0 & 0 & \cdots & 0 & 0 & q_{n-1} & \cdots & q_0 
	\end{pmatrix},
	\]
	having determinant zero, $\operatorname{det}(S(p_A,p'_A))=0$ (see e.g.~\cite{akritas1993sylvester}). This is a polynomial equation in the coefficients $p_n,\ldots, p_0$, $p'_{n-1},\ldots, p'_0$ which themselves are polynomials in the entries $a_{ij}$ of $A$. Writing $S(A) \coloneq S(p_A,p'_A)$, we thus have 
	$$\{A:~ A \text{ does not have n distinct eigenvalues}\} \subseteq \{A:~ \operatorname{det}(S(A)) = 0\}.$$
	When $A$ is real, the later set is the zero set of the polynomial $\operatorname{det}(S(A)): \mathbb{R}^{n^2}\to\mathbb{R}$, and thus has $n^2$-dimensional real Lebesgue measure zero (this is a well-known result but can be hard to locate in the literature, see e.g. \cite{zerosetpol} for an elementary proof).
	
	When $A$ has complex entries, we may view $\operatorname{det}(S(A))=0$ as two real polynomial equations, $\Re{\operatorname{det}(S(A))}=0$ and $\Im{(\operatorname{det}(S(A)))}=0$. The set of solutions for each of these equations has $2n^2$-dimensional Lebesgue measure zero, thus so has their intersection.
\end{proof}

\section{Higher order derivatives}\label{sec4}

Having shown in the previous two sections that minimizing $f_{\mathcal{A}}$ is a well-posed approach to the joint diagonalization problem, at least with probability one, the next natural question is how to do this in practice. While we do not intend to discuss any particular algorithms in this paper, clearly having access to the gradient and the Hessian is of key importance for this task. Rather surprisingly, it seems that the higher order derivatives of $f_{\mathcal{A}}$ have not been computed in previous contributions on the topic, so we shall now focus on this.

For a general order $j$ we shall compute a representation of the $j$'th order differential of our functional $f_\mathcal{A}$, which we later use to obtain the more familiar gradient and Hessian representation of the first and second derivatives. To this end, we rely on standard differential calculus for functions from one normed vector-space to another, see e.g.~\cite{cartan1983differential}.

To compute the higher order derivatives of \eqref{functional}, we first restrict attention to the action of similarity transformation, which, for a given matrix $A$, can be expressed as the function $h_A: GL(n,\mathbb{C})\to \mathbb{M}(n,\mathbb{F})$ defined as
$$h_A(Q) = Q^{-1}AQ.$$

The $j$'th order differential $d^j h_A|_Q$ of $h_A$ at $Q$ is realized as a symmetric $j$-multilinear operator, that is, a function
$$d^j h_A|_Q: GL(n,\mathbb{F})^j \to \mathbb{M}(n,\mathbb{F}),$$
that is linear in each of its $j$ matrix arguments, and invariant under permutations of the arguments. Such an operator is uniquely determined by its values over the diagonal subspace $\{(Z,\ldots, Z);\, Z\in GL(n,\mathbb{F})^k\}$, due to the symmetry property \cite{cartan1983differential}. With slight abuse of notation we shall write $d^k h_A|_Q(Z)$ to mean $d^k h_A|_Q(Z,\ldots,Z)$
Letting $[X,Y]=XY-YX$ denote the matrix commutator product, we have: 

\begin{lemma}\label{l1}
	The $j$'th differential of $h_A$, evaluated on the diagonal, is given by
	$$d^j h_A|_Q(Z) = (-1)^jj![Q^{-1}Z,(Q^{-1}Z)^{j-1}h_A(Q)].$$
\end{lemma}
\begin{proof}
Let $J\geq 1$ and recall that
$(I-X)^{-1} = \sum_{j=0}^J X^j + \mathcal{O}(\|X\|^{J+1})$
as $\|X\|\to 0$. Thus, fixing $Q\in GL(n,\mathbb{F})$, we have that
\begin{align*}
	&(Q+Z)^{-1} = Q^{-1}(I+ZQ^{-1})^{-1} = \sum_{j=0}^J Q^{-1}\left(-ZQ^{-1}\right)^j + \mathcal{O}(\|ZQ^{-1}\|^{J+1})\\
	&= \sum_{j=0}^J (-1)^jQ^{-1}\left(ZQ^{-1}\right)^j + \mathcal{O}(\|Z\|^{J+1})
\end{align*}
as $\|Z\|\to 0$, $Z\in M(n,\mathbb{F})$. We can then compute,
\begin{align*}
	& h_A(Q+Z)=(Q+Z)^{-1}A(Q+Z) \\
	&= \left(\sum_{j=0}^J (-1)^jQ^{-1}\left(ZQ^{-1}\right)^j + \mathcal{O}(\|Z\|^{J+1})\right)A(Q+Z)\\
	&= Q^{-1}AQ + \sum_{j=1}^J (-1)^j\left(Q^{-1}\left(ZQ^{-1}\right)^j AQ - Q^{-1}\left(ZQ^{-1}\right)^{j-1} AZ\right)+\mathcal{O}(\|Z\|^{J+1})\\
	&= h_A(Q) + \sum_{j=1}^J (-1)^j [Q^{-1}Z,(Q^{-1}Z)^{j-1}h_A(Q)] + \mathcal{O}(\|Z\|^{J+1}).
\end{align*}
The error term $\mathcal{O}(\|Z\|^{J+1})$ shows that the above is the $J$'th order Taylor expansion for $h_A$. The uniqueness theorem for Taylor polynomials ensures that the $j$'th term then must equal the $j$'th differential $d^j h_A|_Q(Z)$, see e.g.~\cite{cartan1983differential}. 
\end{proof}
We return to the functional $f_\mathcal{A}$, whose $j$'th differential $d^jf_\mathcal{A}|_Q$  at $Q$ is a $j$-linear operator into $\mathbb{R}$, also known as a $j$-linear form.
Note that $\mathbb{M}(n,\mathbb{F})$ is a finite-dimensional Hilbert space over the $\mathbb{F}$, where the scalar product is induced by the Frobenius norm, i.e.~$\left\langle A,B\right\rangle=\sum_{n_1,n_2}A_{n_1,n_2}\overline{B_{n_1,n_2}}$. However, we note that $\mathbb{M}(n,\mathbb{F})$ can also be seen as a Hilbert space over the reals with the scalar product $\left\langle A,B\right\rangle_{\mathbb{R}}:=\mathsf{Re}\left\langle A,B\right\rangle$.
 As before, we shall abbreviate $d^jf_{\mathcal{A}}|_Q(Z,\ldots,Z)$ by $d^jf_{\mathcal{A}}|_Q(Z)$ and moreover we adopt the shorthand notation $h_k\coloneq h_{A_k}$. 

\begin{proposition}\label{fdiff}
	We have
	$$d^j f_\mathcal{A}|_Q(Z) = \frac{1}{2}\sum_{k=1}^K \sum_{l=0}^j {j \choose l} \left\langle d^{j-l}h_k|_Q(Z),J\diamond d^lh_k|_Q(Z)\right\rangle_{\mathbb{R}}.$$ In particular,
	\begin{align}
		df_{\mathcal{A}}|_Q(Z) =& \sum_{k=1}^K \left\langle  [h_k(Q),Q^{-1}Z],J\diamond h_k(Q)\right\rangle_{\mathbb{R}} \label{d1}\\
		d^2f_{\mathcal{A}}|_Q(Z) =& \sum_{k=1}^K \|J\diamond [h_k(Q),Q^{-1}Z]\|^2+\label{d2}\\&2\left\langle [Q^{-1}Z,Q^{-1}Zh_k(Q)],J\diamond h_k(Q)\right\rangle_{\mathbb{R}}\nonumber 
	\end{align}
\end{proposition}
\begin{proof}
 Using the representation
	$$f_\mathcal{A}(Q) = \frac{1}{2}\sum_{k=1}^K \|J\diamond h_k(Q)\|^2 = \frac{1}{2}\sum_{k=1}^K \left\langle h_k(Q),J\diamond h_k(Q)\right\rangle_{\mathbb{R}},$$
	and Lemma \ref{l1} we have, for fixed $N\geq 1$, that
	\begin{align*}
		f_\mathcal{A}(Q+Z) &= \sum_{k=1}^K \frac{1}{2}\|J\diamond h_k(Q+Z)\|^2\\
		&= \sum_{k=1}^K \frac{1}{2}\left\langle h_k(Q+Z),J\diamond h_k(Q+Z)\right\rangle_{\mathbb{R}} \\
		&= \sum_{k=1}^K \frac{1}{2} \sum_{j,m=0}^N \frac{1}{j!\,m!}\left\langle d^j h_k|_Q(Z),J\diamond d^mh_k|_Q(Z)\right\rangle_{\mathbb{R}} +\mathcal{O}(\|Z\|^{N+1})\\
		&=  \frac{1}{2} \sum_{j=0}^N \frac{1}{j!} \sum_{k=1}^K\sum_{l=0}^j \frac{j!}{l!(j-l)!} \left\langle d^{j-l}h_k|_Q(Z),J\diamond d^lh_k|_Q(Z)\right\rangle_{\mathbb{R}} + \mathcal{O}(\|Z\|^{N+1}),
	\end{align*}
	where we only sum up to $N$ in the last expression since all other terms are of order $\geq N+1$ and are thus included in the ordo-term. As in Lemma \ref{l1} we have obtained the Taylor expansion of $f_\mathcal{A}$ from which we can read out the differentials. The first and second order formulas now follow immediately by combining the general formula with Lemma \ref{l1}.
\end{proof}

The above expressions give the derivatives of $f_\mathcal{A}$ as multilinear forms, but for algorithmic purposes it is more desirable to have them in standard form as the gradient and Hessian, which we show how to retrieve in the coming sections.

\section{The gradient}\label{sec5}

We remind the reader that the gradient $Y$ of a function $g:\mathbb{M}(n,\mathbb{R})\rightarrow\mathbb{R}$ at some point $X$ in a real Hilbert space is uniquely defined as the element in $\mathbb{M}(n,\mathbb{R})$ such that \begin{equation}\label{1}g(X+Z)=g(X)+\left\langle Y,Z\right\rangle_{\mathbb{R}}+o(\|Z\|)\end{equation}
for $Z$ in a neighborhood of 0. 

To treat $\mathbb{M}(n,\mathbb{C})$, we stress again that this also can be seen as a finite dimensional Hilbert space over the reals with real scalar product $\left\langle Y,Z\right\rangle_{\mathbb{R}}=\mathsf{Re}\left\langle Y,Z\right\rangle$.
Therefore, given $g:\mathbb{M}(n,\mathbb{F})\rightarrow\mathbb{R}$ we have that its gradient, if it exists, is the unique matrix $Y\in \mathbb{M}(n,\mathbb{F})$ that satisfies 
$$g(X+Z)=g(X)+\left\langle Y,Z\right\rangle_{\mathbb{R}}+o(\|Z\|).$$
For easier readability we introduce the notation $Q^{-*}\coloneq (Q^{-1})^*$.
\begin{theorem}\label{t1}
Fix $Q$ and set $D_k=Q^{-1}A_k Q$. Then $$\nabla f_\mathcal{A}|_Q=\sum_{k=1}^K {Q^{-*}}\big[D_k^*,D_k\diamond J\big].$$
\end{theorem}
\begin{proof}
Note that, for any matrices $A,B,C\in \mathbb{M}(n,\mathbb{F})$ we have that $\left\langle AB,C \right\rangle = \left\langle A,B^*C \right\rangle=\left\langle B,A^*C \right\rangle$ and thus the same formula is valid also for the corresponding real part. This in combination with Proposition \ref{fdiff} gives
	\begin{align*}
		f_\mathcal{A}(Q+Z) &= f_\mathcal{A}(Q) + df_\mathcal{A}|_Q(Z) + \mathcal{O}(\|Z\|^2)\\
		&=  f_\mathcal{A}(Q) + \sum_{k=1}^K \left\langle  [h_k(Q),Q^{-1}Z],J\diamond h_k(Q)\right\rangle_{\mathbb{R}} + \mathcal{O}(\|Z\|^2)\\
		&=  f_\mathcal{A}(Q) + \sum_{k=1}^K  \left\langle  Q^{-1}Z,[h_k(Q)^*,J\diamond h_k(Q)]\right\rangle_{\mathbb{R}} + \mathcal{O}(\|Z\|^2)\\
		&= f_\mathcal{A}(Q) + \sum_{k=1}^K  \left\langle  Z,{Q^{-*}}[h_k(Q)^*,J\diamond h_k(Q)]\right\rangle_{\mathbb{R}} + \mathcal{O}(\|Z\|^2)\\
		&= f_\mathcal{A}(Q) +  \left\langle Z,\sum_{k=1}^K {Q^{-*}}[D_k^*,J \diamond D_k]\right\rangle_{\mathbb{R}} + \mathcal{O}(\|Z\|^2),
	\end{align*}
which was to be shown, since $D_k=h_k(Q)$.
\end{proof}

A few comments. First we note that since $f_\mathcal{A}$ is homogenous of degree 0, we always have $\nabla f_{\mathcal{A}}|_Q\perp Q$. This in particular means that for any gradient based optimization scheme giving rise to a sequence of points $(Q_{m})_{m=1}^\infty$, the normalized sequence $(Q_{m}/\|Q_m\|)_{m=1}^\infty$ has the same functional values and can be viewed as the retraction of the original sequence to the unit sphere in $\mathbb{M}(n,\mathbb{F})$. Since the latter is compact, this is useful for proving convergence results.

Moreover, note that if we let $\mathcal{D}$ denote the tuple $(D_1,\ldots,D_K)$, then $f_{\mathcal{A}}(Q+Z)=f_{\mathcal{D}}(I+Q^{-1}Z).$ From this observation it is straight-forward to prove that \begin{equation}\label{der}\nabla f_{\mathcal{A}}|_Q=Q^{-*}\nabla f_{\mathcal{D}}|_I,\end{equation}
which also can be verified by the formula in Theorem \ref{t1}. From a theoretical perspective, this fact is maybe not overly interesting, but the different viewpoints have a great impact on algorithmic performance, see \cite{ctw2024joint} for more details.

The formula for the gradient is not new, it has also been computed in the case when $Q$ is restricted to various manifolds \cite{absil2006joint,afsari2004some}, and for a general matrix Lie group in \cite{hori1999joint}. 
However, it seems that this has gone by mainly unnoticed, since none of the existing algorithms for joint diagonalization makes use of the gradient (again, see \cite{ctw2024joint}). 
Moreover, the expression for the Hessian has not been obtained in any of the previous works on the topic, and will be established by the results in the next section. 

\section{The Hessian}\label{sec6}
We now tackle the Hessian. Let $g$ be some (sufficiently smooth) function on a real Hilbert space $\mathcal{X}$ with values in $\R$. One way to look at the Hessian of $g$ at $Q$ is to consider it as a bilinear form $\mathcal{H}|_Q:\mathcal{X}^2\rightarrow \mathbb{R}$ such that \begin{equation}\label{2}g(Q+Z)=g(Q)+\left\langle \nabla f_{\mathcal{A}}|_Q,Z\right\rangle+\frac{1}{2}\mathcal{H}|_Q(Z,Z)+o(\|Z\|^2).\end{equation}
There are many bilinear forms that fulfill the above criterion, but only one if we add the requirement that $\mathcal{H}|_Q$ be symmetric, i.e.~that \begin{equation}\label{herm}
	\mathcal{H}(Z,W)={\mathcal{H}(W,Z)}.\end{equation}

However, just as the gradient of a function $g$ on any Hilbert space $\mathcal{X}$ is an element of $\mathcal{X}$ itself, the Hessian is usually realized as a matrix.
In the more abstract setting 
of
real Hilbert spaces, this matrix corresponds to a symmetric linear operator $H|_Q:\mathcal{X}\rightarrow \mathcal{X}$ such that $\mathcal{H}|_Q(Z,W)=\left\langle H|_Q(Z),W\right\rangle$. Indeed, if $\mathcal{X}=\mathbb{R}^n$ then the matrix representation of  $H|_Q$ is the standard Hessian. Thus, if we want to find an expression for the Hessian but want to avoid introducing a basis in the space of real $n\times n$ matrices (and corresponding $(n^2)\times(n^2)$ matrix-realization of the Hessian), all we need to do is to find a linear operator $H|_Q$ on $\mathcal{X}$ such that \begin{equation}\label{3}g(Q+Z)=g(Q)+\left\langle \nabla g|_Q,Z\right\rangle+\frac{1}{2}\left\langle H|_Q(Z),Z\right\rangle+o(\|Z\|^2)\end{equation}
and $\left\langle H|_Q(Z),W\right\rangle={\left\langle H|_Q(W),Z\right\rangle}$ holds for all matrices $Z,W$. 

In the setting of $\mathcal{X}=\mathbb{M}(n,\mathbb{R})$ and $\mathcal{X}=\mathbb{M}(n,\mathbb{C})$ (regarded as a space over the reals), \eqref{3} becomes
\begin{equation}\label{33}g(Q+Z)=g(Q)+\left\langle \nabla g|_Q,Z\right\rangle_{\mathbb{R}}+\frac{1}{2}\left\langle H|_Q(Z),Z\right\rangle_{\mathbb{R}}+o(\|Z\|^2)\end{equation}
where now $H|_Q:\mathbb{M}(n,\C)\rightarrow \mathbb{M}(n,\C)$ is an operator that is linear \textit{only over the reals} and \begin{equation}\label{5}\left\langle H|_Q(Z),W\right\rangle_{\mathbb{R}}={\left\langle H|_Q(W),Z\right\rangle_{\mathbb{R}}}\end{equation} holds for all matrices $Z,W$.

We remark that both (complex) linear and conjugate linear operators over $\C$ in $\mathbb{M}(n,\mathbb{C})$ give rise to real linear operators when we identify $\mathbb{M}(n,\mathbb{C})$ with $\mathbb{M}(n,\mathbb{R})\times \mathbb{M}(n,\mathbb{R})$ in the obvious way, (i.e.~we identify a complex matrix $X+iY$ with the pair $(X,Y)$, where $X,Y\in\mathbb{M}(n,\mathbb{R})$). As mentioned before, upon introducing a basis in $\mathbb{M}(n,\mathbb{R})$, all real linear operators on $\mathbb{M}(n,\mathbb{R})$ can then be identified as $(n^2)\times (n^2)$ matrices. However, it is worth pointing out that, using the same basis in $\mathbb{M}(n,\mathbb{C})$ we naturally get an identification with $\mathbb{R}^{2n^2}$, where the real part corresponds to the first half and the imaginary part to the latter. In such a basis, real linear operators can then be realized as $(2n^2)\times (2n^2)$-matrices, and it is easy to see that an operator is complex linear if and only if this matrix expression has the form $\left(
\begin{array}{cc}
	A & -B \\
	B & A \\
\end{array}
\right)
$ whereas it is conjugate linear if and only if it has a matrix representation of the form $\left(
\begin{array}{cc}
	A & B \\
	B & -A \\
\end{array}
\right)
$. Moreover,  we see that a symmetric complex linear operator necessarily satisfies $A^T=A$ and $B^T=-B$ whereas a symmetric conjugate linear operator satisfies $A^T=A$ and $B^T=B$. 
It is interesting to note that the below expression for the Hessian is a combination of both complex and conjugate linear operators, in the complex case.

\begin{theorem}\label{t2}
Fix $Q$ and set $D_k=Q^{-1}A_k Q$. Then the Hessian operator at $Q$ is given by \begin{align*}H|_Q(Z)=\sum_{k=1}^K {Q^{-*}}\Big(&\big[D_k^*,J\diamond [D_k,Q^{-1}Z]\big]+\\ &\big[J\diamond D_k,(Q^{-1}Z)^*\big]D_k^* +\left[J\diamond D_k,\left(Q^{-1}ZD_k\right)^*\right]\Big)\end{align*}
\end{theorem}
\begin{proof}
From \eqref{d2} we already know that
$$\left\langle Z,H|_Q(Z)\right\rangle_{\mathbb{R}} = \sum_{k=1}^K \|J\diamond [h_k(Q),Q^{-1}Z]\|^2+2\left\langle [Q^{-1}Z,Q^{-1}Zh_k(Q)],J\diamond h_k(Q)\right\rangle_{\mathbb{R}}.$$
To obtain the Hessian operator giving rise to this, we first note that 
\begin{align*}
	\|J\diamond [h_k(Q),Q^{-1}Z]\|^2 &= \left\langle [h_k(Q),Q^{-1}Z],J\diamond [h_k(Q),Q^{-1}Z] \right\rangle_{\mathbb{R}} \\
	&= \left\langle Q^{-1}Z,\big[h_k(Q)^*,J\diamond [h_k(Q),Q^{-1}Z]\big] \right\rangle_{\mathbb{R}} \\
	&= \left\langle Z,{Q^{-*}}[h_k(Q)^*,J\diamond \big[h_k(Q),Q^{-1}Z]\big]\right\rangle_{\mathbb{R}},
\end{align*}
and it is easy to see that the operator $M_k(Z)\coloneq Q^{-1^*}\big[h_k(Q)^*,J\diamond [h_k(Q),Q^{-1}Z]\big]$ is symmetric, by calculations similar to the above. For the second term we get
\begin{align*}
	\left\langle [Q^{-1}Z,Q^{-1}Zh_k(Q)],J\diamond h_k(Q)\right\rangle_{\mathbb{R}} &= \left\langle Q^{-1}Z,\left[J\diamond h_k(Q),\left(Q^{-1}Zh_k(Q)\right)^*\right]\right\rangle_{\mathbb{R}}\\
	&= \left\langle Z,{Q^{-*}}\left[J\diamond h_k(Q),\left(Q^{-1}Zh_k(Q)\right)^*\right]\right\rangle_{\mathbb{R}}.
\end{align*}
The operator $N_k(Z)\coloneq {Q^{-*}}\left[J\diamond h_k(Q),\left(Q^{-1}Zh_k(Q)\right)^*\right]$ is conjugate linear. Its adjoint can be computed as 
\begin{align*}
	\left\langle X,N_k(Y) \right\rangle_{\mathbb{R}} &=\left\langle X, {Q^{-*}}\left[J\diamond h_k(Q),\left(Q^{-1}Yh_k(Q)\right)^*\right]\right\rangle_{\mathbb{R}}\\
	&= \left\langle h_k(Q)\left[Q^{-1}X,(J\diamond h_k(Q))^*\right]Q^{-1},Y^*\right\rangle_{\mathbb{R}} \\
	&= \left\langle Q^{-*}\left[J\diamond h_k(Q),(Q^{-1}X)^*\right]h_k(Q)^*,Y\right\rangle_{\mathbb{R}} \\
	&=\left\langle N_k^*(X),Y\right\rangle_{\mathbb{R}}
\end{align*}
By adding together this operator and its adjoint we obtain a symmetric conjugate linear operator
$$(N_k^*+N_k)(Z) = {Q^{-*}}[J\diamond h_k(Q),(Q^{-1}Z)^*]h_k(Q)^* + {Q^{-*}}[J\diamond h_k(Q),\left(Q^{-1}Zh_k(Q)\right)^*],$$
that satisfies
\begin{align*}
	&\left\langle Z,(N_k^*+N_k)(Z)\right\rangle_{\mathbb{R}}=2\left\langle [Q^{-1}Z,Q^{-1}Zh_k(Q)],J\diamond h_k(Q)\right\rangle_{\mathbb{R}}, 
\end{align*}
i.e.~it equals the second term in the first identity of the proof.
The Hessian (real linear) operator is thus given (uniquely) by 
$$H(Z) = \sum_{k=1}^K M_k(Z) + N_k^*(Z)+N_k(Z),$$ which written out equals the expression that was to be shown.
\end{proof}

\section{Remarks on the self-adjoint case}\label{sec7}

In this section we address the case in which the matrices in $\mathcal{A}$ are self-adjoint (i.e.~$A_k\in \mathbb{H}(n,\mathbb{F})$ for $k=1,\ldots, K$). 
It is then natural to only consider unitary matrices $Q$, i.e.~such that $Q^*=Q^{-1}$. The collection of such matrices is known as the unitary group (orthogonal group in the real case), which we shall denote by $\mathbb{U}(n,\mathbb{F})$. We therefore introduce $g_{\mathcal{A}}:\mathbb{U}(n,\mathbb{F})\rightarrow [0,\infty)$ by restricting $f_{\mathcal{A}}$ to $\mathbb{U}(n,\mathbb{F})$, and note that this is given by the simpler expression \begin{equation}\label{functionalSA}
g_{\mathcal{A}}(Q)=\frac{1}{2}\sum_{k=1}^K\sum_{i\neq j} \left|(Q^{*}A_kQ)_{ij}\right|^2. 
\end{equation} 
Our key result, Theorem \ref{diverg}, clearly applies to this case unaltered. The result from Section \ref{sec3} need to be modified though, but this is completely straight-forward. One first notes that $(\mathbb{H}(n,\mathbb{F}))^K$  is a linear space and hence we can equip it with a unique (up to a multiplicative constant) Lebesgue measure, which we can use to define absolutely continuous probability measures. Armed with this, it is straight forward to adopt Proposition \ref{probone}, we skip the details. Combined with Theorem \ref{diverg}, we thus have \begin{proposition}\label{probone2}
Let $\mathcal{P}$ be an absolutely continuous probability measure on 
$\big(\mathbb{H}(n,\mathbb{F}\HW{)}\big)^K$, 
and assume that $\mathcal{A}$ is drawn from the corresponding distribution. Then with probability one we have that $\lim_{Q\to Z}g_{\mathcal{A}}(Q) = \infty$ holds at all nonzero rank-deficient points $Z$. 
\end{proposition}
We skip the details of the proof. The only remaining issue worth discussing is how to compute the gradient for this situation.
It is well known that the collection of unitary matrices is a differentiable manifold (see e.g.~\cite{gallier2020differential}, Ch. 2), whose tangent space at the identity equals the set of skew-symmetric matrices $\mathbb{S}(n,\mathbb{F})$, i.e.~matrices $T$ satisfying $T^*=-T$. It is then easy to see that $Q\mathbb{S}(n,\mathbb{F})$ equals the tangent space of $\mathbb{U}(n,\mathbb{F})$ at $Q$, and moreover the formula \eqref{der} in this case simplifies to 
\begin{equation}\label{der2}\nabla f_{\mathcal{A}}|_Q=Q \nabla f_{\mathcal{D}}|_I.\end{equation}
However, when $\mathcal{D}$ is self-adjoint (which it will be if $\mathcal{A}$ is self-adjoint and $Q$ is unitary), it luckily turns out that we already have $\nabla f_{\mathcal{D}}|_I\in\mathbb{S}(n,\mathbb{F})$, and hence we conclude that
\begin{proposition}
Let $\mathcal{A}$ be self-adjoint and $Q$ unitary. Then the gradient of $g_{\mathcal{A}}$, as an element in the tangent space of $\mathbb{U}(n,\mathbb{F})$ at $Q$, is given by
\begin{equation*}\nabla g_{\mathcal{A}}|_Q=Q \nabla f_{\mathcal{D}}|_I=Q\sum_{k=1}^K \big[D_k^*,D_k\diamond J\big].\end{equation*}
\end{proposition}


In order to use the above gradient for algorithmic design on $\mathbb{U}(n,\mathbb{F})$, an obstacle arises from the fact that upon adding a multiple of $-\nabla g_{\mathcal{A}}|_{Q}$ to a matrix $Q$, as is done e.g.~in gradient descent, we inevitably  leave $\mathbb{U}(n,\mathbb{F})$. It is therefore relevant to recall the little known theorem by J.~Keller \cite{keller1975closest} which states that, given any matrix $Q\in GL(n,\mathbb{F})$, the unique closest point on $\mathbb{U}(n,\mathbb{F})$ is given by $\mathcal{R}(Q)=UV^*$ where $U\Sigma V^*$ is the standard singular value decomposition of $Q$. Thus, one may construct descent algorithms on $\mathbb{U}(n,\mathbb{F})$ by iterating, in the simplest case, $Q_{m+1}=\mathcal{R}(Q_{m}-\lambda_k \nabla g_{\mathcal{A}}|_{Q_m})$ for some step-length $\lambda_k$. We refer to \cite{absil2008optimization}, Ch.~2, for a deeper study and further references, of such algorithms on matrix manifolds.

\bibliographystyle{plain}
\bibliography{JEVD}

\end{document}